\documentclass{mathscan}
\usepackage{amsmath}
\usepackage{amsrefs}
\usepackage[all]{xy}

\date{22 April 2010} 

\newtheorem{thm}{Theorem}

\newtheorem{cor}[thm]{Corollary}
\newtheorem{lemma}[thm]{Lemma}
\newtheorem{prop}[thm]{Proposition}

\theoremstyle{definition}

\newtheorem{remark}[thm]{Remark}

\def\mathcs{C^{*}}
\newcommand{\cs}{\ensuremath{\mathcs}}

\DeclareMathSymbol{\rtimes}{\mathbin}{AMSb}{"6F}
\newcommand{\ib}{im\-prim\-i\-tiv\-ity bi\-mod\-u\-le}

\newcommand{\sme}{\,\mathord{\mathop{\text{--}}\nolimits_{\relax}}\,}
\def\ibind#1{\mathop{#1\mathord{\mathop{\text{--}}}}\!\Ind\nolimits}
\def\R{\mathsf{R}}
\def\C{\mathsf{C}}

\DeclareMathOperator{\Ind}{Ind}

\DeclareMathOperator{\Prim}{Prim}

\def\set#1{\{\,#1\,\}}
\newcommand\sset[1]{\{#1\}}
\let\tensor=\otimes
\def\restr#1{|_{{#1}}}

\newbox\hidebox
\def\spechide#1{\setbox\hidebox=\hbox{$#1$}
\hbox to\wd\hidebox{$\box\hidebox^\wedge$\hss}}
\makeatletter
\def\labelenumi{\textnormal{(\@alph\c@enumi)}}
\def\theenumi{\@alph \c@enumi}
\def\labelenumii{\textnormal{(\@roman\c@enumii)}}
\def\theenumii{\@roman \c@enumii}
\newcount\charno
\def\alphapart#1{\charno=96
\advance\charno by#1\char\charno}

\makeatother
\def\<{\langle}
\def\>{\rangle}
\let\ipscriptstyle=\scriptscriptstyle
\def\lipsqueeze{{\mskip -3.0mu}}
\def\ripsqueeze{{\mskip -3.0mu}}
\def\ipcomma{\nobreak\mathrel{,}\nobreak}
\newbox\ipstrutbox
\setbox\ipstrutbox=\hbox{\vrule height8.5pt
width 0pt}
\def\ipstrut{\copy\ipstrutbox}
\def\lip#1<#2,#3>{\mathopen{\relax_{\ipstrut\ipscriptstyle{
#1}}\lipsqueeze
\langle} #2\ipcomma #3 \rangle}
\def\blip#1<#2,#3>{\mathopen{\relax_{\ipstrut
\ipscriptstyle{ #1}}\lipsqueeze\bigl\langle} #2\ipcomma #3 \bigr\rangle}
\def\rip#1<#2,#3>{\langle #2\ipcomma #3
\rangle_{\ripsqueeze\ipstrut\ipscriptstyle{#1}}}
\def\brip#1<#2,#3>{\bigl\langle #2\ipcomma #3
\bigr\rangle_{\ripsqueeze\ipstrut\ipscriptstyle{#1}}}
\def\angsqueeze{\mskip -6mu}
\def\smangsqueeze{\mskip -3.7mu}
\def\trip#1<#2,#3>{\langle\smangsqueeze\langle #2\ipcomma #3
\rangle\smangsqueeze\rangle_{\ripsqueeze\ipstrut\ipscriptstyle{#1}}}
\def\btrip#1<#2,#3>{\bigl\langle\angsqueeze\bigl\langle #2\ipcomma
#3
\bigr\rangle
\angsqueeze\bigr\rangle_{\ripsqueeze\ipstrut\ipscriptstyle{#1}}}
\def\tlip#1<#2,#3>{\mathopen{\relax_{\ipstrut\ipscriptstyle{
#1}}\lipsqueeze \langle\smangsqueeze\langle} #2\ipcomma #3
\rangle\smangsqueeze\rangle}
\def\btlip#1<#2,#3>{\mathopen{\relax_{\ipstrut\ipscriptstyle{
#1}}\lipsqueeze
\bigl\langle\angsqueeze\bigl\langle} #2\ipcomma #3
\bigr\rangle\angsqueeze\bigr\rangle}

\def\ip(#1|#2){(#1\mid #2)}
\def\bip(#1|#2){\bigl(#1 \mid #2\bigr)}
\def\Bip(#1|#2){\Bigl( #1 \bigm| #2 \Bigr)}
\IfFileExists{mathrsfs.sty}{\usepackage{mathrsfs}}
{\let\mathscr\mathcal}
\newcommand{\bundlefont}[1]{\mathscr{#1}}
\newcommand{\A}{\bundlefont A}
\newcommand{\B}{\bundlefont B}

\newcommand\CC{\bundlefont C}
\newcommand\E{\bundlefont E}
\newcommand{\go}{G^{(0)}}
\def\sa_#1(#2;#3){\Gamma_{#1}(#2;#3)}
\newcommand\prima{\Prim A}
\newcommand\clsp{\overline{\operatorname{span}}}

\newcommand\usc{up\-per semi\-con\-tin\-uous}
\newcommand\pb{\bar p}
\newcommand\Bb{\bundlefont Q}
\renewcommand\Bbb{Q}

\newcommand\I{\mathcal{I}}
\newcommand\gcgb{\sa_{c}(G;\B)}
\newcommand\gccgb{\sa_{cc}(G;\B)}

\newcommand\uH{\underline{H}}
\let\phi\varphi

\expandafter\ifx\csname BibSpec\endcsname\relax\else
\BibSpec{collection.article}{%
    +{}  {\PrintAuthors}                {author}
    +{,} { \textit}                     {title}
    +{.} { }                            {part}
    +{:} { \textit}                     {subtitle}
    +{,} { \PrintContributions}         {contribution}
    +{,} { \PrintConference}            {conference}
    +{}  {\PrintBook}                   {book}
    +{,} { }                            {booktitle}
    +{,} { }                            {series}
    +{,} { \voltext}                    {volume}
    +{,} { }                            {publisher}
    +{,} { }                            {organization}
    +{,} { }                            {address}
    +{,} { \PrintDateB}                 {date}
    +{,} { pp.~}                        {pages}
    +{,} { }                            {status}
    +{,} { \PrintDOI}                   {doi}
    +{,} { available at \eprint}        {eprint}
    +{}  { \parenthesize}               {language}
    +{}  { \PrintTranslation}           {translation}
    +{;} { \PrintReprint}               {reprint}
    +{.} { }                            {note}
    +{.} {}                             {transition}
}
\BibSpec{article}{%
    +{}  {\PrintAuthors}                {author}
    +{,} { \textit}                     {title}
    +{.} { }                            {part}
    +{:} { \textit}                     {subtitle}
    +{,} { \PrintContributions}         {contribution}
    +{.} { \PrintPartials}              {partial}
    +{,} { }                            {journal}
    +{}  { \textbf}                     {volume}
    +{}  { \PrintDatePV}                {date}
    +{,} { \eprintpages}                {pages}
    +{,} { }                            {status}
    +{,} { \PrintDOI}                   {doi}
    +{,} { available at \eprint}        {eprint}
    +{}  { \parenthesize}               {language}
    +{}  { \PrintTranslation}           {translation}
    +{;} { \PrintReprint}               {reprint}
    +{.} { }                            {note}
    +{.} {}                             {transition}
}
\BibSpec{book}{%
    +{}  {\PrintPrimary}                {transition}
    +{,} { \textit}                     {title}
    +{.} { }                            {part}
    +{:} { \textit}                     {subtitle}
    +{,} { \PrintEdition}               {edition}
    +{}  { \PrintEditorsB}              {editor}
    +{,} { \PrintTranslatorsC}          {translator}
    +{,} { \PrintContributions}         {contribution}
    +{,} { }                            {series}
    +{,} { \voltext}                    {volume}
    +{,} { }                            {publisher}
    +{,} { }                            {organization}
    +{,} { }                            {address}
    +{,} { pp.~}                        {pages}
    +{,} { \PrintDateB}                 {date}
    +{,} { }                            {status}
    +{}  { \parenthesize}               {language}
    +{}  { \PrintTranslation}           {translation}
    +{;} { \PrintReprint}               {reprint}
    +{.} { }                            {note}
    +{.} {}                             {transition}
}
\fi

\allowdisplaybreaks

\begin{document}
\title[Morita Equivalence for Fell Bundle \cs-algebras]{\boldmath A Classic 
Morita Equivalence Result for Fell Bundle \cs-algebras}

\author{Marius Ionescu}

\address{Department of Mathematics \\ University of Connecticut,
  Storrs, CT 06269-3009}

\email{ionescu@math.uconn.edu}
\curraddr{Department of Mathematics\\ Colgate University \\ Hamilton,
  New York 13346} 

\author{Dana P. Williams}
\address{Department of Mathematics \\ Dartmouth College \\ Hanover, NH
03755-3551}

\email{dana.williams@Dartmouth.edu}

\keywords{Morita equivalence, Fell bundles, transitive groupoids}

\subjclass{46L05, 46L55}

\thanks{This research was supported by the Edward Shapiro fund at
  Dartmouth College.}

\begin{abstract}
  We show how to extend a classic Morita Equivalence Result of Green's
  to the \cs-algebras of Fell bundles over transitive groupoids.
  Specifically, we show that if $p:\B\to G$ is a saturated
    Fell bundle over a transitive groupoid $G$ with stability group
    $H=G(u)$ at $u\in \go$, then $\cs(G,\B)$ is Morita equivalent to
    $\cs(H,\CC)$, where $\CC=\B\restr H$. As an application, we show
    that if $p:\B\to G$ is a Fell bundle over a group $G$ and if there
    is a continuous $G$-equivariant map $\sigma:\Prim A\to G/H$, where
    $A=B(e)$ is the \cs-algebra of $\B$ and $H$ is a closed subgroup,
    then $\cs(G,\B)$ is Morita equivalent to $\cs(H,\CC^{I})$ where
    $\CC^{I}$ is a Fell bundle over $H$ whose fibres are $A/I\sme
    A/I$-\ib s and $I=\bigcap\set{P:\sigma(P)=eH}$.  Green's result
    is a special case of our application to bundles over groups.
\end{abstract}

\maketitle
\section*{Introduction}
\label{sec:intro}

One of the many fundamental results in Green's seminal work
\cite{gre:am78} on \cs-dynamical systems is his theorem
(\cite{gre:am78}*{Theorem~17}) which says that if $(A,G,\alpha)$ is a
dynamical system and if $\sigma:\prima \to G/H$ is a continuous
$G$-equivariant map, then $A\rtimes_{\alpha}G$ is Morita equivalent to
$A/I\rtimes_{\alpha^{I}}H$, where
$I=\bigcap\set{P\in\prima:\sigma(P)=eH}$.  This result is of
particular importance in studying the Mackey machine for regular or
smooth crossed products --- see \cite{wil:crossed}*{Proposition
  8.7~and Theorem~8.16} --- and consequently is a basic component of
the Mackey machine in general.  Our first goal in this note
  is to show 
how Green's result can be formulated for Fell bundles where it is a
straightforward application of the Equivalence Theorem
\cite{muhwil:dm08}*{Theorem~6.4}.  However, recovering Green's
dynamical system version from the Fell bundle version is nontrivial.
Fortuitously, doing so
leads to an interesting application to Fell bundles over groups  which
is our second main result.

This note is a continuation of \cite{ionwil:xx09a}.  In particular, we
will refer to the first section of that paper for basic notation,
conventions and some fundamental facts about Fell bundles.  In
particular, in order that we can apply the Equivalence Theorem from
\cite{muhwil:dm08}, we are going to want to assume that our Fell
bundles are saturated and separable in the sense that the
underlying groupoids are second countable and the underlying Banach
bundles are separable.

\section{The Main Results}
\label{sec:main-results}

The first order of business is to formulate Green's theorem for Fell
bundles.  Then, as mentioned above, we have some work to do to extract
the ``group version''.

\subsection{The Fell Bundle Version}
\label{sec:fell-bundle-version}

\begin{thm}
  \label{thm-fell-me}
  Let $p:\B\to G$ be a separable saturated Fell bundle over a
  \emph{transitive} locally compact groupoid $G$.   If $u\in
  \go$ and if $H:=G(u)=\set{x\in G:s(x)=u=r(x)}$ is the stability
  group at $u$, then $\CC=p^{-1}(H)$ is
  a Fell bundle over $H$, and $\cs(G,\B)$ is
  Morita equivalent to $\cs(H,\CC)$.
\end{thm}
\begin{remark}
  \label{rem-sep}
  Since we have assumed that $p:\B\to G$ is \emph{separable}, this
  implies that $G$ is second countable.  Therefore $G_{u}=s^{-1}(u)$
  is a $(G,H)$-equivalence under the hypotheses of
  Theorem~\ref{thm-fell-me} (see
  \cite{mrw:jot87}*{Theorem~2.2B}).\footnote{The only issue in showing
    that $G_{u}$ is an equivalence is to prove that
    $r_{G}\restr{G_{u}}$ is an open map.  (Note that this can fail if
    $G$ is not second countable: consider $\R$ with the discrete
    topology acting on $\R$ with the usual topology by translation.)
    The point of \cite{mrw:jot87}*{Theorem~2.2B} is to see that this
    map is always open if $G$ is second countable.  Nowadays, a better
    reference for this is Ramsay's \cite{ram:jfa90}*{Theorem~2.1}.}
  In particular, $\cs(G)$ and $\cs(H)$ are Morita equivalent by the
  Equivalence Theorem for groupoids (\cite{mrw:jot87}*{Theorem~2.8}).
\end{remark}

\begin{proof}
  Let $\CC=p^{-1}(H)$ and $p_{H}=p\restr{\CC}$.  Then it is
  straightforward to see that $p_{H}:\CC\to H$ is a Fell bundle over
  $H$.  Let $\E=p^{-1}(G_{u})$ and $q=p\restr {\E}$.  We will show
  that $q:\E\to G_{u}$ is a $(\B,\CC)$-equivalence as in
  \cite{muhwil:dm08}*{Definition~6.1}.  Then Theorem~\ref{thm-fell-me}
  will follow from the Equivalence Theorem
  \cite{muhwil:dm08}*{Theorem~6.4}.

  Since $G_{u}$ is a $(G,H)$-equivalence (Remark~\ref{rem-sep}) and
  $q:\E\to G_{u}$ is clearly an \usc\ Banach bundle, we just need to
  verify axioms (a), (b) and (c) of
  \cite{muhwil:dm08}*{Definition~6.1}.  To do this, first observe that
  $\B$ and $\CC$ act on the left and right, respectively, on $\E$ via
  restriction of the multiplication in $\B^{(2)}$; then the axioms
  (a), (b) and (c) for an action given in the second paragraph of
  \cite{muhwil:dm08}*{\S6} are clearly satisfied.\footnote{Notice that
    there is a typo in axiom~(c): it should read $\|b\cdot e\|\le
    \|b\|\|e\|$.  Of course, this follows from
    \cite{ionwil:xx09a}*{Lemma~1}  in our case.}  Then axiom~(a)
  of \cite{muhwil:dm08}*{Definition~6.1} follows from the
  associativity of multiplication in $\B^{(2)}$.

  For axiom (b) of \cite{muhwil:dm08}*{Definition~6.1}, we define
  $\lip\B<\cdot, \cdot>:\E*_{s}\E\to \B$ by
  \begin{equation*}
    \lip\B<b,c>=bc^{*},
  \end{equation*}
  and $\rip\CC<b,c>:\E*_{r}\E\to \CC$ by
  \begin{equation*}
    \rip\CC<b,c>=b^{*}c.
  \end{equation*}
  Then it is not hard to check that properties (i)--(iv) hold.  For
  example, if $(b,c)\in\E*_{s}\E$, then
  $[q(b),q(c)]=q(b)q(c)^{-1}=p(bc^{*})$.

  As for axiom~(c), let
  $A=\sa_{0}(\go;\B)$ the \cs-algebra of $\B$ over $\go$.
  We are given that $E(x)=B(x)$ is a $A(r(x))\sme
  A(u)$-\ib.  But $A(u)$ is the \cs-algebra of $\CC$ over $H^{(0)}=u$.
  Thus, (c) holds and $q:\E\to G_{u}$ is an
  equivalence.\footnote{There is, sadly, also a misprint in part~(c)
    of \cite{muhwil:dm08}*{Definition~6.1}: it should read that ``each
    $E(t)$ is a $B(r(t))\sme C(s(t))$-\ib.''}  This completes the
  proof.
\end{proof}

\subsection{The Group Version}
\label{sec:group-version}

It is hardly obvious that Green's Theorem for \cs-dynamical systems is
a consequence of Theorem~\ref{thm-fell-me}. In fact, showing this
requires a fair bit of gymnastics.  We will obtain Green's
  result as a special case of a result for Fell bundles over groups
  (Theorem~\ref{thm-grp-thm}) that is of considerable interest in its
own right.

Let $p:\B\to G$ be a Fell bundle over a locally compact
\emph{group}. (Note that this implies that the underlying Banach
bundle is continuous rather than merely upper semicontinuous --- see
\cite{ionwil:xx09a}*{Remark~3}.)  Suppose that $\sigma:\prima\to G/H$
is an equivariant map, where $A$ is the \cs-algebra $A(e)$ of $\B$
over $G^{(0)}=\sset e$, and $H$ is a closed subgroup of $G$.  We let
\begin{equation}\label{eq:20}
  I:=\bigcap \set{P\in\prima:\sigma(P)=eH}.
\end{equation}
As in Theorem~\ref{thm-fell-me}, we let $\CC=p^{-1}(H)$.  Then $I$ is
a $H$-invariant ideal in the \cs-algebra $A=A(e)$ of $\CC$.  We adopt
the notations and constructions of \cite{ionwil:xx09a}*{\S3.1}\
applied to the Fell bundle $p\restr{\CC}:\CC\to G$ and the invariant
ideal $I$.  In particular, for each $h\in H$, $C(h)$ is an $A\sme
A$-\ib, and $C_{I}(h):=C(h)\cdot I$ is an $I\sme I$-\ib\ by
\cite{ionwil:xx09a}*{Lemma~10 and Proposition~14}.  Thus by
\cite{rw:morita}*{Proposition~3.25}, the quotient
$C^{I}(h):=C(h)/C_{I}(h)$ is an $A/I\sme A/I$-\ib, and by
\cite{ionwil:xx09a}*{Proposition~15}, $\CC^{I}:=\coprod_{h\in
  H}C^{I}(h)$ has a natural topology making it into a Fell bundle over
$H$ with the operations induced from $\CC$.

\begin{thm}
  \label{thm-grp-thm}
  Let $p:\B\to G$ be a separable saturated Fell bundle over a locally
  compact group $G$ such that there is a continuous $G$-equivariant
  map $\sigma:\prima \to G/H$, where $A$ is the \cs-algebra of $\B$
  and $H$ is a closed subgroup of $G$.  If $I$ is the ideal of $A$
  given in \eqref{eq:20}, and if $\CC$ is the Fell bundle $p^{-1}(H)$
  as above, then $\cs(G,\B)$ is Morita equivalent to $\cs(H,\CC^{I})$
\end{thm}
\begin{remark}
  \label{rem-green-at-last}
  Notice that if $\B=A\times G$ is the Fell bundle associated to the
  dynamical system $(A,G,\alpha)$, then $\CC^{I}$ is the Fell bundle
  $A/I\times H$ associated to $(A/I,H,\alpha^{I})$.  Therefore Green's
  Theorem is a special case of Theorem~\ref{thm-grp-thm}.
\end{remark}


To prove Theorem~\ref{thm-grp-thm} we want to appeal to
Theorem~\ref{thm-fell-me}.  We do this by first building a
  Fell bundle $\Bb$ over the transitive transformation groupoid
  $G\times G/H$ such that if $u=(e,eH)\in (G\times G/H)^{(0)}$ and if
  $\uH=\set{(h,eH):h\in H}$ is the stability group at $u$, then
  $\Bb\restr{\uH}$, which we view as a Fell bundle over $H$, is naturally
  identified with $\CC^{I}$ (see Proposition~\ref{prop-fb1}). Since
  Theorem~\ref{thm-fell-me} implies $\cs(G\times G/H,\Bb)$ is Morita
  equivalent to $\CC^{I}$, we complete the proof of
  Theorem~\ref{thm-grp-thm} by showing that $\cs(G\times G/H,\Bb)$ is
  isomorphic to $\cs(G,\B)$.  We do this in the next section as
  Proposition~\ref{prop-phi-iso}.

 \section{The isomorphism}
 \label{sec:isomorphism}

 Suppose that we are given a Fell bundle $p:\B\to G$ over a locally
 compact \emph{group} $G$.  Let $A=B(e)$ be the \cs-algebra over $e$.
 In \cite{ionwil:xx09a}*{Proposition~9}, we showed that $\prima$ is a
 $G$-space.  Suppose that there is a $G$-equivariant map
 $\sigma:\prima\to G/H$ for a closed subgroup $H$ of $G$.  In this
 section, we want to show that there is a naturally associated Fell
 bundle $\pb:\Bb\to G\times G/H$ over the transformation groupoid
 $G\times G/H$ such that $\cs(G,\B)$ and $\cs(G\times G/H,\Bb)$ are
 isomorphic (Proposition~\ref{prop-phi-iso}).

 \begin{lemma}
   \label{lem-I-ideals}
   Let $\sigma:\prima\to G/H$ be a continuous $G$-equivariant map as
   above.  Then $A$ is a $C_{0}(G/H)$-algebra (as in
   \cite{wil:crossed}*{Proposition~C.5}) and $A(xH)=A/I(xH)$, where
   \begin{equation*}
     I(xH):=\bigcap \set{P\in\prima:\sigma(P)=xH}.
   \end{equation*}
 \end{lemma}
 \begin{proof}
   By \cite{wil:crossed}*{Proposition~C.5} and preceding discussion,
   $A$ is a $C_{0}(G/H)$-algebra with fibres $A(xH)=A/J_{xH}$, where
   \begin{equation*}
     J_{xH}=\clsp\set{\phi\cdot a:\text{$a\in A$, $\phi\in C_{0}(G/H)$
         and $\phi(xH)=0$}}.
   \end{equation*}
   Thus, we just need to confirm that $J_{xH}=I(xH)$.  However, if
   $a(P)$ denotes the image of $a\in A$ in the quotient $A/P$, then
   for all $P\in \prima$,
   \begin{equation}
     \label{eq:9}
     (\phi\cdot a)(P)=\phi\bigl(\sigma(P)\bigr) a(P)
   \end{equation}
   (see the discussion preceding
   \cite{wil:crossed}*{Proposition~C.5}).  If $P\supset J_{xH}$, then
   since the left-hand side of \eqref{eq:9} vanishes for all $a\in A$
   and $\phi\in C_{0}(G/H)$ with $\phi(xH)=0$, we see that
   $\sigma(P)=xH$.  On the other hand, if $\sigma(P)=xH$, then using
   \eqref{eq:9}, we see that $J_{xH}\subset P$.  Therefore
   $J_{xH}=I(xH)$ as required.
 \end{proof}

 \begin{cor}
   \label{cor-hxmap}
   For each $x\in G$, let $h_{x}:\I(A)\to\I(A)$ be the Rieffel
   homeomorphism $\ibind{B(x)}$ induced by the \ib\ $B(x)$ (see
   \cite{rw:morita}*{Proposition~3.24}).  Then
   $h_{x}\bigl(I(yH)\bigr)=I(xyH)$.
 \end{cor}
 \begin{proof}
   Since $\sigma$ is continuous, $\sigma^{-1}(yH)$ is closed in
   $\prima$.  In particular, $P\supset I(yH)$ if and only if
   $\sigma(P)=yH$.  Since $h_{x}$ is containment preserving and has
   inverse $h_{x^{-1}}$ by \cite{rw:morita}*{Theorem~3.29}, $P\supset
   I(yH)$ if and only if $h_{x}(P)\supset h_{x}\bigl(I(yH)\bigr)$.
   But equivariance means that $\sigma\bigl(h_{x}(P)\bigr) =x\cdot
   \sigma(P)$.  Thus
   \begin{align*}
     h_{x}\bigl(I(yH)\bigr) &= \bigcap\set{h_{x}(P):\sigma(P)=yH}
     = \bigcap \set{P:\sigma(P)=xyH} \\
     &= I(xyH).\qedhere
   \end{align*}
 \end{proof}

 \begin{remark}
   \label{rem-cohen}
   As noted in \cite{ionwil:xx09a}*{Remark~5}, if $X$ is a
   $A\sme B$-\ib\ and $J$ is an ideal in $A$, then the Cohen
   Factorization Theorem implies that
   \begin{equation*}
     \clsp\set{a\cdot x:\text{$a\in J$ and $x\in X$}}=\set{a\cdot
       x:\text{$a\in J$ and $x\in X$}}.
   \end{equation*}
   Consequently, we write simply $J\cdot X$ for the above $A\sme
   B$-submodule.  Similarly, we'll write $X\cdot I$ for the
   corresponding $A\sme B$-submodule when $I$ is an ideal in $B$.
 \end{remark}

 For each $x\in G$, $B(x)$ is an $A\sme A$-\ib.  Since $I(yH)$ and
 $I(xyH)$ are matched up by the Rieffel correspondence, the following
 is a consequence of basic Morita theory (see
 \cite{rw:morita}*{Propositions 3.24~and 3.25}).
 \begin{cor}
   \label{cor-I-switching}
   Let $x,y\in G$.  Then
   \begin{enumerate}
   \item $B(x)\cdot I(yH)=I(xyH)\cdot B(x)$,
   \item $B(x) \cdot I(yH)$ is a $I(xyH)\sme I(yH)$-\ib,
     and
   \item $B(x)/B(x)\cdot I(yH)$ is $A(xyH)\sme
     A(yH)$-\ib.
   \end{enumerate}
 \end{cor}

 It follows immediately from Corollary~\ref{cor-I-switching} that the
 \emph{Banach space}
 \begin{equation*}
   \Bbb(x,yH):=B(x)/B(x)\cdot I(x^{-1}yH) = B(x)/I(yH)\cdot B(x)
 \end{equation*}
 is an $A(yH)\sme A(x^{-1}yH)$-\ib.  Note that the $A(yH)$-valued
 inner product on $\Bbb(x,yH)$ is given by taking the appropriate
 quotient of the left $A$-valued inner product on $B(x)$ which is
 given by the Fell bundle multiplication. Thus
 \begin{equation}
   \label{eq:12}
   \blip A(yH)< [b],[c]>= b c^{*}(yH),
 \end{equation}
 where $[b]$ denotes the image of $b\in B(x)$ in $\Bbb(x,yH)$, and
 $bc^{*}(yH)$ is the image of $bc^{*}$ in $A(yH)$.

 Let
 \begin{equation*}
   \Bb:=\coprod_{(x,yH)\in G\times G/H}\Bbb(x,yH)
 \end{equation*}
 and $\pb:\Bb\to G\times G/H$ be the associated bundle.  Naturally, we
 want to equip $\Bb$ with a topology and operations making it into a
 Fell bundle over the transformation groupoid $G\times G/H$.  This
 will take a bit of work and will be accomplished in
 Proposition~\ref{prop-fb1}.

 If $f\in\sa_{c}(G;\B)$, then let $\Phi(f)(x,yH)$ denote the image of
 $f(x)$ in $\Bbb(x,yH)$.

 \begin{lemma}
   \label{lem-key-bb}
   If $f\in \gcgb$, then $(x,yH)\mapsto
   \bigl\|\Phi(f)(x,yH)\bigr\|$ is upper
   semicontinuous and vanishes at infinity on $G\times G/H$.
 \end{lemma}
 \begin{proof}
   Since $\Bbb(x,yH)$ is an \ib, and in view of \eqref{eq:12}, we have
   \begin{equation*}
     \|\Phi(f)(x,yH)\|^{2}=\|f(x)f(x)^{*}(yH)\|.
   \end{equation*}
   Therefore it suffices to show that $(x,yH)\mapsto \|a(x)(yH)\|$ is
   upper semicontinuous and vanishes at infinity for $a\in
   C_{c}(G,A)$.

   To show upper semicontinuity, it suffices to show that if
   $(x_{i},y_{i}H)\to (x,yH)$ and $\|a(x_{i})(y_{i}H)\|\ge \epsilon>0$
   for all $i$, then we also have $\|a(x)(yH)\|\ge\epsilon$.  If
   $\|a(x)(yH)\|<\epsilon$, then we can find $\phi\in C_{c}(G/H)$ such
   that $\phi$ is identically one in a neighborhood of $yH$ and such
   that $\|\phi\cdot a(x)\|<\epsilon$.  Since we can assume that
   $\phi(y_{i}H)=1$ for all $i$, we certainly have
   $\|a(x_{i})\|\ge\epsilon$.  But this contradicts the fact that
   $x\mapsto \|a(x)\|$ is continuous.

   To see that $(x,yH)\mapsto \|a(x)(yH)\|$ vanishes at infinity,
   suppose that $\|a(x_{i})(y_{i}H)\|\ge\epsilon>0$ for all $i$.  It
   will suffice to see that $\sset{(x_{i},y_{i}H)}$ has a convergent
   subsequence.  Since $a$ has compact support, we can pass to a
   subsequence, relabel, and assume that $x_{i}\to x$.  Then by
   continuity, we can assume that $\|a(x)(y_{i}H)\|\ge \epsilon/2$ for
   large $i$.  Since $A$ is a $C_{0}(G/H)$-algebra, $yH\mapsto
   \|a(x)(yH)\|$ must vanish at infinity
   \cite{wil:crossed}*{Proposition~C.10(a)}.  Therefore $
   \sset{y_{i}H}$ must have a convergent subsequence. 
   This completes
   the proof.
 \end{proof}

 As in \cite{kmqw:xx09}*{Lemma~1.2} and
 \cite{ionwil:xx09a}*{\S3.2}, there is a nondegenerate
 homomorphism $\iota:A\to M\bigl(\cs(G,\B)\bigr)$ such that
 $\iota(a)f(x)=a f(x)$.  Then
 \begin{equation*}
   \Phi(\iota(a)f)(x,yH)=a(yH)\cdot
 \Phi(f)(x,yH).
 \end{equation*}
 Since $\iota$ is nondegenerate, it extends to
 $M(A)$, and by composition with the $C_{0}(G/H)$-structure map of
 $C_{0}(G/H)$ into the center of $M(A)$, we get a map
 $\hat\iota:C_{0}(G/H)\to M\bigl(\cs(G,\B)\bigr)$.  We'll write
 $\psi\cdot f$ in place of $\hat\iota(\psi)f$.  Note that
 \begin{equation*}
   \Phi(\psi\cdot f)(x,yH)=\psi(yH)\Phi(f)(x,yH).
 \end{equation*}

 \begin{remark}
   \label{rem-approximation}
   If $p:\B\to X$ is an \usc-Banach bundle and $f$ is a not
   necessarily continuous section such that there are $f_{i}\in
   \sa_{0}(X;\B)$ converging uniformly to $f$, then $f\in
   \sa_{0}(X;\B)$.
 \end{remark}

 \begin{lemma}
   \label{lem-gghaction}
   If $f\in \sa_{c}(G;\B)$ and $\theta\in C_{0}(G\times G/H)$, then
   there is a section $\theta\cdot f\in \sa_{c}(G;\B)$ such that
   \begin{equation*}
     \Phi(\theta\cdot f)(x,yH)=\theta(x,yH)\Phi(f)(x,yH).
   \end{equation*}
 \end{lemma}
 \begin{proof}
   Notice that if $b\in B(x)$, then
   \begin{equation}\label{eq:14}
     \|b\|^{2}=\|bb^{*}\|=\sup_{yH\in G/H}\|bb^{*}(yH)\| =\sup_{yH\in
       G/H}\|b(yH)\|^{2}, 
   \end{equation}
   where $b(yH)$ denotes the image of $b$ in the quotient
   $\Bbb(x,yH)$. (Note that with this notation, $\Phi(f)(x,yH)$ can
   also be written as $f(x)(yH)$.)  Also, if
   $\theta(x,yH)=(\omega\tensor\phi)(x,yH):=\omega(x)\phi(yH)$, for
   $\omega\in C_{0}(G)$ and $\phi\in C_{0}(G/H)$, then we can define a
   continuous section $\theta\cdot f$ by $\theta\cdot
   f(x):=\omega(x)(\phi\cdot f)(x)$ (because scalar multiplication is
   continuous from $\C\times\B\to \B$).  Now suppose that we have a
   finite sum $\sum_{i}\omega_{i}\tensor \phi_{i}$ of such functions.
   Then
   \begin{equation}\label{eq:15}
     \begin{split}
       \|\sum_{i} (\omega_{i}\tensor\phi_{i})\cdot f\|&= \sup_{x\in
         G}\|\sum_{i} \omega_{i}(x) \phi_{i}\cdot f(x)\| \\
       &= \sup_{x\in G}\sup_{yH\in
         G/H}\|\sum_{i}\omega_{i}(x)\phi_{i}(yH)
       f(x)(yH)\| \\
       &\le \|\sum_{i} \omega_{i}\tensor \phi_{i}\|_{\infty}
       \sup_{x\in
         G}\sup_{yH\in G/H}\|f(x)(yH)\| \\
       &= \|\sum_{i} \omega_{i}\tensor \phi_{i}\|_{\infty} \|f\|.
     \end{split}
   \end{equation}
   This shows that $\theta\cdot f$ is a well defined element of
   $\sa_{c}(G;\B)$ provided $\theta\in C_{0}(G\times G/H)$ is a finite
   sum as above.

   If we let $\Bb(x):=\coprod_{yH\in G/H} \Bbb(x,yH)$, then using
   \cite{ionwil:xx09a}*{Theorem~2}\ and \eqref{eq:14}, we get an
   \usc-Banach bundle $\pb_{x}:\Bb(x)\to G/H$ such that $\Gamma:=
   \set{yH\mapsto b(yH):b\in B(x)}$ are continuous sections in
   $\sa_{0}(G/H;\Bb(x))$. Since $f\in \Gamma$ implies $\phi\cdot f\in
   \Gamma$ for any $\phi\in C_{0}(G/H)$, it follows that $\Gamma$ is
   uniformly dense in $\sa_{0}(G/H;\Bb(x))$ (see
   \cite{muhwil:dm08}*{Lemma~A.4} for example).  Therefore
   $\Phi_{x}:B(x)\to \sa_{0}(G/H;\Bb(x))$, defined by
   $\Phi_{x}(b)(yH)=b(yH)$, is an isometric isomorphism.  In
   particular, if $b\in B(x)$ and $\phi\in C_{0}(G/H)$ then there is a
   $\phi\cdot b\in B(x)$ such that $(\phi\cdot b)(yH)=\phi(yH)b(yH)$.

   Now suppose that $\theta$ is an arbitrary element of $C_{0}(G\times
   G/H)$ and $f\in \sa_{c}(G;\B)$.  In view of the above, for each
   $x\in G$, there is a $b(x)\in B(x)$ such that
   \begin{equation*}
     b(x)(yH)=\theta(x,yH)f(x)(yH)\quad\text{for all $yH\in G/H$.}
   \end{equation*}
   So it only remains to see that $x\mapsto b(x)$ is continuous.

   Let $\sset{\theta_{i}}$ be a sequence of functions, which are
   elementary sums as in \eqref{eq:15}, such that $\theta_{i}\to
   \theta$ uniformly.  Note that there are $g_{i}\in \sa_{c}(G;\B)$
   such that
   \begin{equation*}
     g_{i}(x)(yH)=\theta_{i}(x,yH)f(x)(yH)\quad\text{for all $x\in G$ and
       $yH\in G/H$.}
   \end{equation*}
   Computing just as in \eqref{eq:15}, we have
   \begin{equation*}
     \|g_{i}(x)-b(x)\|\le \|\theta_{i}-\theta\|_{\infty}\|f(x)\|.
   \end{equation*}
   Thus, by Remark~\ref{rem-approximation}, $x\mapsto b(x)$ defines an
   element of $\sa_{c}(G;\B)$ as required, and the lemma
     is proved.
 \end{proof}

 Now let
 \begin{equation*}
   \gccgb:=\set{f\in\gcgb:\text{$\Phi(f)$ vanishes off a compact set in
       $G\times G/H$}}.
 \end{equation*}

 It is a consequence of Lemmas \ref{lem-key-bb}~and
 \ref{lem-gghaction} that $\Gamma:= \set{\Phi(f):f\in \gccgb}$
 satisfies the requirements of \cite{ionwil:xx09a}*{Theorem~2}.  Thus
 we can equip $\pb:\Bb\to G\times G/H$ with the structure of an
 \usc-Banach bundle over $G\times G/H$ such that $\Gamma\subset
 \sa_{c}(G\times G/H;\Bb)$.

 \begin{remark}[Comments on Definitions]
   \label{rem-comments}
   We equip $G\times G/H$ with the usual Haar system where we identify
   $(G\times G/H)^{(0)}$ with $G/H$ and define
   $\sset{\lambda^{yH}}_{yH\in G/H}$ by
   \begin{equation*}
     \lambda^{yH}(g)=\int_{G}g(x,yH)\,dx.
   \end{equation*}
   There is however a subtlety in defining $\cs(G;\B)$.  Here, to make
   the proof of Proposition~\ref{prop-fb1} more elegant, we are going
   to treat $G$ as a groupoid.  The point is that then the involution
   on $\gcgb$ is given by $ f^{*}(x)=f(x^{-1})^{*}$.  It is often more
   natural when working with a Fell bundle over a group to use the
   involution used by Fell \& Doran in
   \citelist{\cite{fd:representations1} \cite{fd:representations2}}
   where $f^{*}(x)=\Delta(x^{-1})f(x^{-1})^{*}$ and $\Delta$ is the
   modular function on the group $G$.  For example, when using the
   second formulation, it is much easier to see that one recovers the
   usual group \cs-algebra and crossed product constructions as
   special cases.  Fortunately, the isomorphism class of $\cs(G,\B)$
   is unaffected by our choice --- see \cite{kmqw:xx09}*{Remark~1.5}.
 \end{remark}

 \begin{prop}
   \label{prop-fb1}
   As above, if $b\in \B$, let $\,b(yH)$ be the image of $b$ in
   $\Bbb(p(b),yH)$.  Then if $(m,n)\in
   \Bb^{(2)}=\set{(m,n)\in\Bb\times\Bb: \pb(m)=\pb(n)}$, it follows
   that $(m, n)$ is of the form $(b(yH),b'(p(b)^{-1}yH))$ for $b,b'\in
   \B$.  Then we get a well-defined map from $\Bb^{(2)}$ to $\Bb$ by
   $b(yH)b'(p(b)^{-1}yH):= bb'(yH)$.  We can also get a well defined
   involution from $\Bb$ to $\Bb$ via $b(yH)^{*}=b^{*}(p(b)^{-1}yH)$.
   Then, with respect to these operations, $\pb:\Bb\to G\times G/H$ is
   a Fell bundle. Furthermore, $\Bb\restr {\uH}$ and
     $\CC^{I}$ are isomorphic as Fell bundles over $H$.
 \end{prop}

\begin{proof}
  If $(m,n)\in \Bb^{(2)}$, then $(\pb(m),\pb(n))\in (G\times
  G/H)^{(2)}$.  Thus for appropriate $x,y,z\in G$, we must have
  $\pb(m)=(x,yH)$ and $\pb(n)=(z,x^{-1}yH)$.  Thus we can certainly
  find $b\in B(x)$ and $b'\in B(z)$ such that $b(yH)=m$ and
  $b(x^{-1}yH)=n$.  If $c\in B(x)$ and $c'\in B(z)$ also satisfy
  $c(yH)=m$ and $c'(x^{-1}yH)=n$, then there are $d\in I(y\cdot
  H)\cdot B(x)$ and $d'\in I(x^{-1}yH)\cdot B(z)$ such that $c=b+d$
  and $c'=b'+d'$.  Then
  \begin{align*}
    cc'&= bb'+db'+bd'+b'd' \\
    &\in bb' + I(yH) B(x)B(z) + B(x)I(x^{-1}yH)B(z)+
    I(yH)B(x)I(x^{-1}yH)
    B(z) \\
    \intertext{which, in view of Corollary~\ref{cor-I-switching} as
      well as the observations that $B(x)B(z)=B(xz)$ and
      $I(yH)^{2}=I(yH)$, is in} &bb'+I(yH)\cdot B(xz).
  \end{align*}
  Therefore $cc'(yH)=bb'(yH)$ in $B(xz)/I(yH)\cdot B(xz)$.  Therefore
  multiplication is well-defined.  A similar argument holds for the
  involution.

  To establish continuity of multiplication, we first need to observe
  that if $b_{i}\to b$ in $\B$ and if $y_{i}H\to yH$, then
  \begin{equation}
    \label{eq:1}
    b_{i}(y_{i}H)\to b(yH)\quad\text{in $\Bb$.}
  \end{equation}
  To see this, let $p(b_{i})=x_{i}$ and $p(b)=x$, and let $f\in\gcgb$
  be such that $f(x)=b$.  Since $f(x_{i})-b_{i}\to 0_{x}$, we must
  have
  \begin{equation*}
    \|f(x_{i})-b_{i}\|\to 0.
  \end{equation*}
  But then
  \begin{equation*}
    \|\Phi(f)(x_{i},y_{i}H)- b_{i}(y_{i}H)\|\to 0.
  \end{equation*}
  Since $\Phi(f)(x_{i},y_{i}H)\to \Phi(f)(x,yH)=b(yH)$, it follows from
  \cite{wil:crossed}*{Proposition~C.20} that \eqref{eq:1} holds.

  Now suppose that $(c_{i},c_{i}')\to (c,c')$ in $\Bb^{(2)}$ with
  \begin{equation*}
    \bigl(\pb(c_{i}),\pb(c_{i}')\bigr) =
    \bigl((x_{i},y_{i}H),(z_{i},x_{i}^{-1}y_{i}H)\bigr) \to
    \bigl(\pb(c),\pb(c')\bigr) = \bigl((x,yH),(z,x^{-1}yH)\bigr).
  \end{equation*}
  We want to show that $c_{i}c_{i}'\to cc'$.

  Keep in mind that if $f\in\gcgb$, then with our conventions,
  $f(x)(yH)$ and $\Phi(f)(x,yH)$ both denote the image of $f(x)$ in
  $\Bbb(x,yH)$.  In particular, if $f,g\in\gcgb$, then
  $(f(x)g(z))(yH)= \Phi(f)(x,yH)\Phi(g)(z,x^{-1}yH)$.  Thus if
  $f,g\in\gcgb$ are such that $\Phi(f)(x,yH)=c$ and
  $\Phi(g)(z,x^{-1}yH)=c'$, then $f(x_{i})g(z_{i})\to f(x)g(z)$ in
  $\B$.  Then it follows from \eqref{eq:1} that
  \begin{equation*}
    \Phi(f)(x_{i},y_{i}H)\Phi(g)(z_{i},x_{i}^{-1}y_{i}H) \to
    \Phi(f)(x,yH)\Phi(g)(z,x^{-1}yH). 
  \end{equation*}
  Since
  \begin{equation*}
    \|\Phi(f)(x_{i},y_{i}H)-c_{i}\|\to 0\quad \text{and} \quad
    \|\Phi(g)(z_{i}, x_{i}^{-1}y_{i}H)-c_{i}'\|\to 0,
  \end{equation*}
  we can use \cite{ionwil:xx09a}*{Lemma~1} to show that
  \begin{equation*}
    \|\Phi(f)(x_{i},y_{i}H)\Phi(g)(z_{i},x_{i}^{-1}y_{i}H)- c_{i}c_{i}'\|\to 0.
  \end{equation*}
  Then \cite{wil:crossed}*{Proposition~C.20} implies that
  $c_{i}c_{i}'\to cc'$.  Therefore, multiplication is continuous.  The
  continuity of the involution is proved similarly.

  It now follows easily that axioms (a), (b) and (c) of
  \cite{muhwil:dm08}*{Definition~1.1} are satisfied.  Furthermore, if
  $(e,yH)\in (G\times G/H)^{(0)}$, then $\Bbb(e,yH)$ is the
  \cs-algebra $A(yH)$.  So axiom~(d) is also satisfied.  And we have
  already observed that $\Bbb(x,yH)$ is a $A(yH)\sme A(x^{-1}yH)$-\ib.
  Thus all the axioms of \cite{muhwil:dm08}*{Definition~1.1} are
  satisfied and $\pb:\Bb\to G\times G/H$ is a Fell bundle as required.

  To verify the last assertion, we observe that $I(eH)=I$ and
  $\Bbb(h,eH)=B(h)/B(h)\cdot I=C^{I}(h)$.  Thus, both $\Bb\restr{\uH}$
  and $\CC^{I}$ are built from
  \begin{equation*}
    \coprod_{h\in H} C^{I}(h).
  \end{equation*}
  Therefore it suffice to see that the identity map is a
  homeomorphism.  But in any Banach bundle $p:\A\to X$, we have
  $a_{i}\to a$ in $\A$ if and only if for some section $g\in
  \sa_{c}(X;\A)$ with $g(p(a))=a$, we have $\|a_{i}-g(p(a_{i}))\|\to
  0$ (for example, see \cite{muhwil:dm08}*{Lemma~A.3}).  But if $b\in
  C^{I}(h_{0})$, then there is a $f\in\sa_{cc}(G;\B)$ such that
  $[f(h_{0})]=b$.  But $h\mapsto [f(h)] $ is a section of $\CC^{I}$
  and $h\mapsto \Phi(f)(h,eH)=[f(h)]$ is also a section of
  $\Bb\restr{\uH}$.  It now follows easily that the identity map is
  bicontinuous.  This completes the proof of the proposition.
\end{proof}

\begin{lemma}
  \label{lem-cc-dense}
  We have that $\gccgb$ is dense in $\gcgb$ in the inductive limit
  topology.
\end{lemma}
\begin{proof}
  This is a straightforward consequence of
  \cite{muhwil:dm08}*{Lemma~A.4} and \eqref{eq:14}.
\end{proof}

\begin{prop}
  \label{prop-phi-iso}
  The map $\Phi:\gccgb\to \sa_{c}(G\times G/H;\Bb)$ extends to a
  isomorphism of $\cs(G,\B)$ onto $\cs(G\times G/H,\Bb)$.
\end{prop}
\begin{proof}
  First we'll show that $\Phi$ is a $*$-homomorphism.  Then we'll see
  that $\Phi$ is a bijection of $\gccgb$ onto $\sa_{c}(G\times
  G/H;\Bb)$.  We'll finish by showing that $\Phi$ and $\Phi^{-1}$
  are bounded with respect to the universal norms.

  Since $b\mapsto b(yH)$ is a bounded linear map of $B(x)$ onto the
  quotient $B(x)/I(yH)\cdot B(x)$,
  \begin{align*}
    f*g(x)(yH)&=\int_{G} (f(z)g(z^{-1}x))(yH)\,dz \\
    &= \int_{G} f(z)(yH) g(z^{-1}x)(z^{-1}yH)\,dz.
  \end{align*}
  Therefore, $\phi(f*g)(x,yH)=\Phi(f)*\Phi(g)(x,yH)$, and $\Phi$
  preserves multiplication.  Similarly,
  \begin{equation*}
    \Phi(f)^{*}(x,yH)=\Phi(f)(x^{-1},x^{-1}yH)^{*} =
    \bigl(f(x^{-1})(x^{-1}yH)  \bigr)^{*} = f(x^{-1})^{*}(yH),
  \end{equation*}
  while on the other hand,
  \begin{equation*}
    \Phi(f^{*})(x,yH)=f^{*}(x)(yH)=f(x^{-1})^{*}(yH).
  \end{equation*}
  Thus, $\Phi$ is a $*$-homomorphism.  Clearly, $\Phi$ is injective.

  Using Lemma~\ref{lem-gghaction}, we see that
  $\Phi\bigl(\gccgb\bigr)$ is a $C_{0}(G\times G/H)$ module.  Thus
  \cite{muhwil:dm08}*{Lemma~A.4} implies that $\Phi\bigl(\gccgb\bigr)$
  is inductive limit dense in $\sa_{c}(G\times G/H,\Bb)$.  Thus if
  $F\in \sa_{c}(G\times G/H;\Bb)$, then there are $f_{i}\in \gccgb$
  such that $\Phi(f_{i})\to F$ in the inductive limit topology.  For
  each $x\in G$, $yH\mapsto F(x,yH)$ is (not necessarily continuous)
  section of $\pb_{x}:\Bb_{x}\to G/H$ which is uniformly approximated
  by the continuous sections $yH\mapsto
  \Phi(f_{i})(x,yH)=f_{i}(x)(yH)$.  Therefore, by
  Remark~\ref{rem-approximation}, there is a $f(x)\in B(x)$ such that
  $f(x)(yH)=F(x,yH)$.  But $f_{i}$ must converge uniformly to $f$, and
  it follows that $f\in \gccgb$.  But then $\Phi(f)=F$.

  Now we notice that
  \begin{equation*}
    \int_{G} \|\Phi(f)(x,yH)\|\,dx \le \int_{G}\|f(x)\|\,dx.
  \end{equation*}
  Thus,
  \begin{equation*}
    \|\Phi(f)\|_{I}\le \|f\|_{I},
  \end{equation*}
  where the $I$-norms are computed in $\sa_{c}(G\times G/H;\Bb)$ and
  $\gcgb$, respectively.  Let $L$ be a faithful representation of
  $\cs(G\times G/H,\Bb)$.  Then $L\circ \Phi$ is a $I$-norm decreasing
  $*$-homomorphism of $\gccgb$, which must extend to a $I$-norm
  decreasing representation $L'$ of $\gcgb$ since $\gccgb$ is $I$-norm
  dense in $\gcgb$ by Lemma~\ref{lem-cc-dense}.  Therefore
  \begin{equation*}
    \|\Phi(f)\|=\|L(\Phi(f))\|=\|L'(f)\|\le \|f\|,
  \end{equation*}
  and $\Phi$ is norm decreasing for the universal norms.

  But if $\Phi(f_{i})\to \Phi(f)$ in the inductive limit topology on
  $\sa_{c}(G\times G/H;\Bb)$, then $f_{i}\to f$ in the inductive limit
  topology on $\gcgb$.  Thus if $R$ is a faithful representation of
  $\cs(G,\B)$, then $R\circ \Phi^{-1}$ is a $*$-homomorphism $R'$ of
  $\sa_{c}(G\times G/H;\Bb)$ which is continuous in the inductive
  limit topology.  By \cite{muhwil:dm08}*{Remark~4.14}, $R'$ is
  bounded and
  \begin{equation*}
    \|f\|=\|R(f)\|=\|R'(\Phi(f))\|\le \|\Phi(f)\|.
  \end{equation*}
  This completes the proof.
\end{proof}



\def\noopsort#1{}\def\cprime{$'$} \def\sp{^}

\end{document}